\documentclass[11pt]{article}
\usepackage[a4paper,margin=27mm]{geometry}
\usepackage[T1]{fontenc}
\usepackage{lmodern}
\usepackage{microtype}
\usepackage{amsmath,amssymb,amsthm,mathtools}
\usepackage{aliascnt}
\usepackage{enumitem}
\usepackage{xcolor}
\usepackage[colorlinks=true,linkcolor=blue!55!black,
  citecolor=blue!55!black,urlcolor=blue!60!black]{hyperref}
\usepackage[nameinlink,noabbrev,capitalise]{cleveref}

\newtheorem{theorem}{Theorem}[section]
\newaliascnt{lemma}{theorem}
\newtheorem{lemma}[lemma]{Lemma}
\aliascntresetthe{lemma}
\newaliascnt{proposition}{theorem}

\aliascntresetthe{proposition}
\newaliascnt{corollary}{theorem}
\newtheorem{corollary}[corollary]{Corollary}
\aliascntresetthe{corollary}
\newaliascnt{claim}{theorem}

\aliascntresetthe{claim}
\theoremstyle{definition}
\newaliascnt{definition}{theorem}
\newtheorem{definition}[definition]{Definition}
\aliascntresetthe{definition}
\theoremstyle{remark}
\newaliascnt{remark}{theorem}
\newtheorem{remark}[remark]{Remark}
\aliascntresetthe{remark}

\newcommand{\Gnp}{G(n,p)}
\newcommand{\del}{\operatorname{del}}
\newcommand{\calI}{\mathcal I}
\newcommand{\calL}{\mathcal L}
\newcommand{\eps}{\varepsilon}
\newcommand{\e}{\mathrm e}

\title{Sparse Approximate Chromatic Profiles of Triangle-Free Graphs}
\author{Guorong Gao\thanks{School of Mathematics and Statistics, Fuzhou
University, Fuzhou 350108, China. Email: \mbox{grgao@fzu.edu.cn}.}
\and Jialin He\thanks{School of Mathematical Sciences, Key Laboratory of
MEA (Ministry of Education), Shanghai Key Laboratory of PMMP, and Nantong
Institute for Applied Mathematics and Artificial Intelligence, East China
Normal University, Shanghai 200241, China. Email: jlhe@math.ecnu.edu.cn.}}
\date{}

\begin{document}
\maketitle

\begin{abstract}
We prove a sparse version of the four-colour theorem of Brandt and
Thomass\'{e}, answering a question of Allen, B\"ottcher, Kohayakawa and
Roberts.  For every fixed $0<\gamma\le1/10$ and every
$p=p(n)\in(0,1]$, asymptotically almost surely every spanning
triangle-free $H\subseteq G(n,p)$ with
$\delta(H)\ge(1/3+\gamma)pn$ can be made four-partite by deleting at
most $\min\{C_\gamma n/p,(1/8+\gamma)pn^2\}$ edges.
In fact, deleting at most $C_\gamma n/p$ edges yields a graph that
admits a homomorphism to an Andr\'{a}sfai or Vega graph with certificate
complexity at most $1/(3\gamma)$.
Together with matching lower bounds from random blow-ups, this structural
result determines, uniformly in $p$, the minimum-degree thresholds for
$q$-partiteness with $O(n/p)$ edge deletions: $2/5$ for $q=2$, $10/29$ for $q=3$,
and $1/3$ for every fixed $q\ge4$.
For every fixed $q\ge2$ and $\log n/n\ll p\ll n^{-1/2}$,
asymptotically almost surely $G(n,p)$ contains a spanning triangle-free
subgraph with minimum degree $(1-o(1))pn$ that requires
$(1/(2q)+o(1))pn^2$ edge deletions to become $q$-partite, showing
that the coefficient $1/(2q)$ cannot be improved even under this
stronger degree condition.
\end{abstract}



\section{Introduction}\label{sec:introduction}

Minimum-degree conditions impose strong restrictions on the chromatic
number and structure of triangle-free graphs.  One seeks the smallest
value of $\delta(H)/n$ that forces an $n$-vertex triangle-free graph $H$
to admit a colouring with a prescribed number of colours, or with a
number bounded independently of $n$.  A stronger structural conclusion
is a homomorphism to a triangle-free graph of bounded order, which
represents $H$ as a subgraph of a blow-up of a bounded template.

For two colours, the first question has a precise answer.  The theorem of
Andr\'{a}sfai--Erd\H{o}s--S\'{o}s \cite{AES} states that every
$n$-vertex triangle-free graph $H$ with $\delta(H)>2n/5$ is bipartite.
This degree condition is sharp, since blow-ups of the five-cycle with
equal class sizes are triangle-free and non-bipartite, with minimum degree
exactly $2n/5$.
Thus $2/5$ is the minimum-degree threshold for forcing bipartiteness.
For three colours, Jin \cite{Jin} proved that $\delta(H)>10n/29$ guarantees
$\chi(H)\le3$, while H\"aggkvist's construction \cite{Haggkvist} shows
that the constant $10/29$ cannot be lowered.  In this sense, $10/29$ is
the corresponding threshold for three-colourability.

If one asks only for a number of colours bounded independently of $n$,
a weaker minimum-degree condition suffices.  Thomassen \cite{Thomassen}
proved that, for every fixed $\gamma>0$, an $n$-vertex triangle-free graph
$H$ with $\delta(H)\ge(1/3+\gamma)n$ has chromatic number bounded in terms
of $\gamma$ alone.  The structural theorem of \L{}uczak
\cite{LuczakStructure} strengthens this conclusion by showing that $H$
admits a homomorphism to a triangle-free graph of order bounded in terms
of $\gamma$.  Brandt and Thomass\'{e}
\cite[Corollary~4.2]{BrandtThomasse} subsequently proved that the condition
$\delta(H)>n/3$ already guarantees $\chi(H)\le4$.  In
\cite[Corollary~4.1]{BrandtThomasse}, they classify the twin-free weighted
maximal triangle-free graphs of weighted minimum degree greater than
$1/3$.  Identifying vertices with identical neighbourhoods and weighting
each resulting vertex by the proportion of vertices in its class shows
that every maximal triangle-free graph satisfying $\delta(H)>n/3$ is a
blow-up of an Andr\'{a}sfai or Vega graph; see also
\cite{BrandtPisanski,ChenJinKoh,LPRStrong} for the structural development.

We study sparse analogues of these questions for spanning triangle-free
subgraphs $H$ of the binomial random graph $G(n,p)$.  Here the natural
degree scale is $pn$ rather than $n$, and we allow edge deletions before
requiring a colouring or a homomorphism.  The aim is to determine both
the minimum-degree thresholds and the number of deletions needed.
Allen--B\"ottcher--Kohayakawa--Roberts \cite[Theorem~4]{ABKR} proved a
sparse analogue of Thomassen's bounded-colour theorem.  For every
$\gamma>0$, there are constants $C=C(\gamma)>0$ and $r=r(\gamma)$ such
that, for every $p=p(n)\in(0,1]$, asymptotically almost surely (a.a.s.) every
spanning triangle-free $H\subseteq G(n,p)$ with
\[
 \delta(H)\ge(1/3+\gamma)pn
\]
can be made $r$-partite by deleting at most
\[
 \min\left\{C\frac np,
       \left(\frac1{2r}+\gamma\right)pn^2\right\}
\]
edges.  Here we apply their theorem with $\gamma/2$ to allow a
non-strict minimum-degree hypothesis.  Thus their theorem bounds the number of parts in terms of
$\gamma$, but does not give the four-partite conclusion suggested by
Brandt and Thomass\'{e}'s dense theorem.  They explicitly raised the
following question in the introduction of \cite{ABKR}.
\begin{quote}
``It would also be interesting to know whether Theorem~4 could be improved
to generalise the result of Brandt and Thomass\'{e}.  We conjecture that this
is the case.''
\end{quote}
We answer this question by obtaining four parts, with deletion bound
$\min\{C_\gamma n/p,(1/8+\gamma)pn^2\}$, and prove the stronger conclusion
that deleting at most $C_\gamma n/p$ edges yields a homomorphism to a
bounded Andr\'{a}sfai or Vega template.  We also determine the corresponding
minimum-degree thresholds for every prescribed number of colours.

Throughout the paper, $G(n,p)$ denotes the binomial random graph on vertex
set $[n]$, in which every pair of vertices is an edge independently with
probability $p=p(n)\in(0,1]$.  An event holds
\emph{a.a.s.}\ if its
probability tends to one as $n\to\infty$.  All graphs are finite and simple.
For a graph $G$, we write $N_G(x,A)=N_G(x)\cap A$,
$d_G(x,A)=|N_G(x,A)|$, and
$N_G(x,y,A)=N_G(x,A)\cap N_G(y,A)$.  We use $e_G(X,Y)$ for the number of
edges with one endpoint in each of two disjoint sets $X$ and $Y$, and write
$\delta(G)$, $\Delta(G)$, $\alpha(G)$, and $\chi(G)$ for the minimum degree,
maximum degree, independence number, and chromatic number, respectively.
A homomorphism $G\to T$ is an edge-preserving vertex map; equivalently,
$G$ is a subgraph of a blow-up of $T$.  All logarithms are natural.
Asymptotic notation refers to $n\to\infty$; unless indicated otherwise,
its implicit constants may depend on the fixed parameters but not on
$n$ or $p$.  For positive functions $f$ and $g$, we write $f\ll g$
(or $g\gg f$) when $f=o(g)$.

To state the colouring results precisely, we introduce two parameters.
For an integer $q\ge2$, we write
\[
 \del_q(G)=\min_{\phi:V(G)\to[q]}
 \bigl|\{xy\in E(G):\phi(x)=\phi(y)\}\bigr|.
\]
Thus $\del_q(G)$ is the minimum number of edges whose deletion makes $G$
$q$-partite.

For a fixed $q\ge2$, the \emph{sparse edit-chromatic profile value}
$\alpha_q^{\rm edit}$ is the infimum of the real numbers $\alpha$ with
the following property.  For every $\eps>0$ there is $C=C(q,\eps)$ such
that, for every sequence $p=p(n)$, a.a.s.\ every
spanning triangle-free $H\subseteq G(n,p)$ satisfying
\[
 \delta(H)\ge(\alpha+\eps)pn
\]
obeys $\del_q(H)\le Cn/p$.  The constant is uniform in $p$, while the
probability statement is understood for each sequence $p=p(n)$.

The chromatic threshold studied in \cite{SparseChromaticThresholds}
concerns a bound on the chromatic number independent of $n$, without edge
deletions.  By contrast, our edit-chromatic profile fixes the number of
colours and quantifies the required deletions.  This is a sparse
analogue of the dense approximate chromatic profiles studied
by Illingworth \cite[Section~1.1]{Illingworth}.

Our first main result is a sparse analogue of the Brandt--Thomass\'{e}
theorem, with an explicit bound on the complexity of the target graph.

\begin{theorem}
\label{thm:template}
For every $0<\gamma\le1/10$ there is $C=C(\gamma)>0$ such that, for every
$p=p(n)\in(0,1]$, a.a.s.\ the following holds.
Every spanning triangle-free $H\subseteq\Gnp$ with
\[
 \delta(H)\ge(1/3+\gamma)pn
\]
can be transformed, by deleting at most $Cn/p$ edges, into a graph
$H^\circ$ admitting a homomorphism to an
Andr\'{a}sfai or Vega graph\footnote{The graphs and their certificate
complexity $L(T)$ are defined in \cref{sec:compression}.} $T$ satisfying
\[
 L(T)\le \frac1{3\gamma}.
\]
In particular, if $T=\Gamma_k$ then $3k-1\le(3\gamma)^{-1}$, while if
$T=\Upsilon_i^{\mu\nu}$ then
$27i-19-3\mu-3\nu\le(3\gamma)^{-1}$.
\end{theorem}

For each integer $q\ge2$, we use the notation
\[
 \alpha_q=
 \begin{cases}
  2/5,&q=2,\\
  10/29,&q=3,\\
  1/3,&q\ge4.
 \end{cases}
\]
The complexity bound in \cref{thm:template}, the colourability of the
templates, and matching lower-bound constructions give the following
complete sparse edit-chromatic profile.

\begin{corollary}
\label{cor:profile}
For every fixed integer $q\ge2$, we have $\alpha_q^{\rm edit}=\alpha_q$.
Moreover, for every $\gamma>0$ there is $C=C(q,\gamma)>0$ such that, for
every $p=p(n)$, a.a.s.\ every spanning triangle-free
$H\subseteq\Gnp$ with $\delta(H)\ge(\alpha_q+\gamma)pn$ satisfies
\begin{equation}
 \del_q(H)\le
 \min\left\{C\frac np,
       \left(\frac1{2q}+\gamma\right)pn^2\right\}.
 \label{eq:profile-upper}
\end{equation}
\end{corollary}

\begin{remark}
For $q=4$, \cref{cor:profile} gives the conjectured sparse
Brandt--Thomass\'{e} theorem, with coefficient $1/8$.
The case $q=2$ recovers the bipartite conclusion of
\cite[Theorem~3]{ABKR}; the additional colouring conclusions are the
three-partite threshold and the four-partite bound above $1/3$.
The term $Cn/p$ is $o(pn^2)$ precisely when $np^2\to\infty$.
By contrast, the $pn^2$ upper bound follows from random colouring
and concentration of the total number of host edges, with a separate
elementary argument for very small $p$, rather than from the structural
argument.
\end{remark}

We next address the sharpness of \cref{cor:profile}.  The $n/p$ scale in
\eqref{eq:profile-upper} is attained, up to a constant factor, by a
construction of Allen--B\"ottcher--Kohayakawa--Roberts
\cite[Theorem~5]{ABKR}.  For every $\eta>0$ and integer $q\ge2$, there are
$c,c'>0$ such that, whenever $n^{-1/2}/c'\le p\le c'$, a.a.s.\ $\Gnp$ contains a spanning triangle-free graph $H$ with
\[
 \delta(H)>(1/2-\eta)pn
 \quad\text{and}\quad
 \del_q(H)\ge c n/p.
\]
For $0<\gamma<1/2-\alpha_q$, choosing
$0<\eta<1/2-\alpha_q-\gamma$ ensures that $H$ satisfies the minimum-degree
condition in \cref{cor:profile}.  The following theorem shows that the
thresholds $\alpha_q$ cannot be lowered.  It also gives the optimal
coefficient $1/(2q)$ when $\log n/n\ll p\ll n^{-1/2}$.

\begin{theorem}
\label{thm:sharpness}
For each fixed integer $q\ge2$, the following assertions hold.
\begin{enumerate}[label=(\roman*),leftmargin=*]
\item For every $\eta>0$, there is $c_{q,\eta}>0$ such that, whenever
$pn/\log n\to\infty$, a.a.s.\ $\Gnp$ contains a
spanning triangle-free graph $H$ satisfying
\[
 \delta(H)\ge(\alpha_q-\eta)pn
 \quad\text{and}\quad
 \del_q(H)\ge c_{q,\eta}pn^2.
\]
\item For every fixed $0<\gamma<1-\alpha_q$, if
\[
 \frac{pn}{\log n}\longrightarrow\infty,
 \qquad np^2\longrightarrow0,
\]
then a.a.s.\ $\Gnp$ contains a spanning triangle-free
graph $H$ satisfying
\[
 \delta(H)\ge(\alpha_q+\gamma)pn
 \quad\text{and}\quad
 \del_q(H)=\left(\frac1{2q}+o(1)\right)pn^2.
\]
\end{enumerate}
\end{theorem}

When $np^2\to\infty$, \cref{thm:sharpness}(i) gives
$\del_q(H)=\omega(n/p)$, showing that the thresholds in
\cref{cor:profile} cannot be lowered.

Our proof builds on the sparse regularity and bad-star method of
Allen et al.\ \cite{ABKR}.  A vertex profile records the clusters in
which the vertex has many neighbours.  The structural step is to
control edges not only within a profile class, but also between distinct
classes whose profiles intersect.  An orientation argument, together
with a bad-star estimate allowing different parameters at different
centres, bounds all such edges by $O(n/p)$.  The sharper two-sided
inheritance estimate of \cite{ABHKP} keeps the vertex-exception cost
on the same scale.  The surviving edges then join disjoint independent
profiles, so an auxiliary graph built from these profiles and the reduced
graph admits a homomorphism to an Andr\'{a}sfai or Vega graph by the strong
Brandt--Thomass\'{e} theorem \cite{LPRStrong}.  An average against the
weights of a regular blow-up \cite{LPRRT} bounds the certificate
complexity.  The resulting proof separates the sparse edge-deletion
argument from the deterministic classification step.

The paper is organised as follows.  \Cref{sec:preliminaries} defines the
template graphs and their certificate complexity, proves the deterministic
lemmas, and collects the required random-graph tools.  In
\cref{sec:repair}, we obtain a homomorphism by deleting edges.  We deduce
the structural and chromatic results in \cref{sec:main-proofs}, and prove
the lower bounds in \cref{sec:optimality}.

\section{Preliminaries}\label{sec:preliminaries}

We collect the background results and prove the auxiliary lemmas needed
for the main argument.  We first define the template graphs and their
certificate complexity, then establish two deterministic lemmas, and
finally assemble the random-graph tools.

\subsection{Andr\'{a}sfai and Vega graphs}
\label{sec:compression}

We begin with the graph definitions, following
\cite[Sections~3 and~4.1]{LPRRT}.

\begin{definition}
\label{def:templates}
For an integer $k\ge1$, the \emph{Andr\'{a}sfai graph} $\Gamma_k$ has
vertex set $\mathbb Z/(3k-1)\mathbb Z$.  Distinct vertices $r,s$ are
adjacent if and only if $(r-s)\bmod(3k-1)\in\{k,\ldots,2k-1\}$.
Thus $\Gamma_1=K_2$ and $\Gamma_2=C_5$.

For an integer $i\ge2$, the graph $\Upsilon_i^{00}$ is obtained from
$\Gamma_i$, whose residue $j$ is denoted by $v_j$, by adding an induced
six-cycle $a,v,c,u,b,w$ in this order and two adjacent vertices $x,y$.
All eight added vertices are distinct from the vertices $v_j$.
The vertex $x$ is joined to $a,b,c$, the vertex $y$ is joined to $u,v,w$,
and all edges are present between the two sets in each of the following
pairs.
\[
 (\{a,u\},\{v_0,\ldots,v_{i-1}\}),\quad
 (\{b,v\},\{v_i,\ldots,v_{2i-1}\}),\quad
 (\{c,w\},\{v_{2i},\ldots,v_{3i-2}\}).
\]
There are no other edges.  For $\mu,\nu\in\{0,1\}$, the \emph{Vega graph}
$\Upsilon_i^{\mu\nu}$ is obtained from $\Upsilon_i^{00}$ by deleting $y$
if $\mu=1$ and deleting $v_{2i-1}$ if $\nu=1$.
\end{definition}

In particular, $\Upsilon_2^{11}$ is the Gr\"otzsch graph.  All these graphs
are triangle-free; $\Gamma_1$ is bipartite, each $\Gamma_k$ with $k\ge2$ is
three-chromatic, and every Vega graph is four-chromatic
\cite{LPRStrong,LPRRT}.

For these templates, we define the \emph{certificate complexity} by
\begin{equation}
 L(\Gamma_k)=3k-1,\qquad
 L(\Upsilon_i^{\mu\nu})=27i-19-3\mu-3\nu.
 \label{eq:complexity}
\end{equation}
A blow-up of $T$ replaces each vertex $u$ by an independent set $V_u$
and each edge $uv$ by all edges between $V_u$ and $V_v$, with no other
edges.  Unless stated otherwise, the classes $V_u$ may be empty.
The natural map sends every vertex of $V_u$ to $u$.
For $T=\Gamma_k$, the graph $Q=T$ is $k$-regular.  For $T=\Upsilon_i^{\mu\nu}$,
\cite[Fact~4.2]{LPRRT} gives a $\kappa$-regular blow-up $Q$ of $T$,
where $\kappa=9i-(6+\mu+\nu)$.  Every vertex class is nonempty by the
support identity in \cite[equation~(14)]{LPRRT}, and
\cite[equation~(15)]{LPRRT} gives $|V(Q)|=3\kappa-1=L(T)$.
Here the equation numbers refer to the published version.
In both cases, $Q$ has $L(T)$ vertices and degree $(L(T)+1)/3$.
Thus, for Vega graphs, $L(T)$ is the order of $Q$, not the order
$3i+7-\mu-\nu$ of $T$.  These properties will be used in
\cref{lem:certificate}.  In particular, $L(\Gamma_2)=5$ and
$\min_{i,\mu,\nu}L(\Upsilon_i^{\mu\nu})=L(\Upsilon_2^{11})=29$.
The corresponding degree-to-order ratios of $Q$ are $2/5$ and $10/29$.

A \emph{probability vector on a finite set $X$} is a function
$w:X\to[0,1]$ with $\sum_{x\in X}w(x)=1$.  For $A\subseteq X$, we write
$w(A)=\sum_{x\in A}w(x)$.

As in \cite[Definition~1.1]{LPRStrong}, we say that a graph has
property $D_4$ if, for each $h\in\{1,2,3,4\}$ and every sequence of
$3h$ vertices
(repetitions allowed), some vertex is adjacent to at least $h+1$ terms of
the sequence.  We use the strong Brandt--Thomass\'{e} theorem of
\L{}uczak--Polcyn--Reiher \cite[Theorem~1.2]{LPRStrong}, which states that
every maximal
triangle-free $D_4$ graph is a blow-up of an Andr\'{a}sfai or Vega graph
with all vertex classes nonempty.

\subsection{Deterministic lemmas}\label{sec:deterministic}

The following two lemmas provide a homomorphism to a template graph and
bound its certificate complexity using neighbourhood weights.

\begin{lemma}
\label{lem:compression}
For an integer $t\ge1$ and $\rho>0$, suppose that $R$ is a triangle-free
graph on $[t]$ with $\delta(R)>(1/3+\rho)t$ and that $\calI$ is a family
of independent sets of $R$ with $|I|>(1/3+\rho)t$ for every $I\in\calI$.
The graph $F_0$ obtained from $R$ by adding a vertex $z_I$ adjacent to
every vertex of $I$ for each $I\in\calI$, and joining distinct profile
vertices $z_I,z_J$ exactly when $I\cap J=\varnothing$, with no other
edges added, is triangle-free and has property $D_4$.
Moreover, there exist an Andr\'{a}sfai or Vega graph $T$, a
homomorphism $\pi:F_0\to T$, and a probability vector $w$ on $V(T)$ such
that, for every $u\in V(T)$,
\begin{equation}
 w(N_T(u))>1/3+\rho.
 \label{eq:weighted-margin}
\end{equation}
\end{lemma}

\begin{proof}
We call $[t]$ the core and the vertices $z_I$ the profile vertices.
A triangle using no profile vertex would be a triangle in $R$.  A triangle
using one profile vertex would give an edge of $R[I]$.  A
triangle using two profile vertices and one core vertex would require two
disjoint profiles to contain the same core vertex.  Three profile vertices
would give three pairwise disjoint subsets of $[t]$, each larger than
$t/3$.  Thus $F_0$ is triangle-free.

The degree and profile-size assumptions imply that every vertex of
$F_0$ has more than $(1/3+\rho)t$ neighbours in the core.
For each $h\in\{1,2,3,4\}$ and every sequence $x_1,\ldots,x_{3h}$ of
vertices, with repetitions counted separately, we therefore have
\[
 \sum_{v\in[t]}\bigl|\{j\in[3h]:vx_j\in E(F_0)\}\bigr|
 =\sum_{j=1}^{3h}d_{F_0}(x_j,[t])
 >3h(1/3+\rho)t>ht.
\]
Some core vertex is therefore adjacent to at least $h+1$ terms, so
$F_0$ has $D_4$.

Any maximal triangle-free supergraph $F^*$ of $F_0$ on the same vertex
set retains $D_4$, so the strong Brandt--Thomass\'{e} theorem makes
$F^*$ a blow-up of an Andr\'{a}sfai or Vega graph $T$ with nonempty
classes.  The natural map $\pi:F^*\to T$ induces the weights
$w(u)=|\pi^{-1}(u)\cap[t]|/t$ for $u\in V(T)$.
The sets $\pi^{-1}(u)\cap[t]$, $u\in V(T)$, partition $[t]$, so
$\sum_{u\in V(T)}w(u)=1$ and $w$ is a probability vector; some weights
may be zero.  For every $u\in V(T)$, its nonempty class contains a
vertex $x$, all of whose more than $(1/3+\rho)t$ core neighbours in
$F_0$ lie in classes indexed by $N_T(u)$.  Consequently,
$w(N_T(u))\ge d_{F_0}(x,[t])/t>1/3+\rho$, which proves
\eqref{eq:weighted-margin}; restriction of $\pi$ gives the claimed
homomorphism.
\end{proof}

The next lemma bounds the complexity of the target graph using the
neighbourhood weights.  The regular blow-up provides a second probability
vector with constant neighbourhood weight, against which we average the
assumed lower bounds.

\begin{lemma}
\label{lem:certificate}
For $\rho>0$ and an Andr\'{a}sfai or Vega graph $T$, any probability
vector $w$ on $V(T)$ satisfying $w(N_T(u))>1/3+\rho$ for every
$u\in V(T)$ guarantees
\[
 L(T)<\frac1{3\rho}.
\]
\end{lemma}

\begin{proof}
The regular blow-up $Q$ described after \eqref{eq:complexity} has
$L(T)$ vertices, degree $(L(T)+1)/3$, and nonempty vertex classes.
Its natural map $\varphi:Q\to T$ induces the probability vector
$\lambda(u)=|\varphi^{-1}(u)|/L(T)$, which is positive on every vertex.
For each $v\in V(T)$ and any $x\in\varphi^{-1}(v)$, the identity
$N_Q(x)=\varphi^{-1}(N_T(v))$ gives
$\lambda(N_T(v))=(L(T)+1)/(3L(T))$.
The $\lambda$-weighted average of the assumed inequalities, together
with the symmetry of adjacency in $T$, now gives
\begin{align*}
 \frac13+\rho
 <\sum_{v\in V(T)}\lambda(v)w(N_T(v))
   =\sum_{u\in V(T)}w(u)\lambda(N_T(u))
 =\frac{L(T)+1}{3L(T)}\sum_{u\in V(T)}w(u)
   =\frac13+\frac1{3L(T)}.
\end{align*}
Hence $L(T)<(3\rho)^{-1}$.
\end{proof}

\subsection{Random-graph tools}\label{sec:random-inputs}

We collect the random-graph properties used in \cref{sec:repair}.
These comprise degree bounds for typical vertices, two-sided regularity inheritance,
codegree bounds for typical edges, and a bound on disjoint bad stars.
Parts (i) and (iii) of \cref{lem:host} follow from
\cite[Lemmas~14 and~15]{ABKR}, and part (ii) uses the sharper two-sided
inheritance bound in \cite[Lemma~1.27]{ABHKP}.  In part (iv), we adapt
\cite[Lemma~16]{ABKR} to sets of any fixed positive linear size and allow
the badness parameter to vary between stars.  We give the details of this
adaptation and then state a sparse regularity lemma that provides both
a triangle-free reduced graph and a vertexwise degree bound for the same
partition.

As in \cite[Section~2]{ABKR}, we call a pair of disjoint vertex sets
$(X,Y)$ \emph{$(\eps,d,p)$-lower-regular in a graph $G$} if every
$X'\subseteq X$ and $Y'\subseteq Y$ with $|X'|>\eps|X|$ and
$|Y'|>\eps|Y|$ satisfy
\[
 e_G(X',Y')>(d-\eps)p|X'||Y'|.
\]
Here \emph{two-sided regularity inheritance} means that, for all but a
controlled number of vertices $x$, restricting both sides of a
lower-regular pair to their neighbourhoods of $x$ in the host graph
$\Gamma$ preserves lower-regularity in $G$, with a larger error parameter,
as made precise in \cref{lem:host}(ii).

\begin{lemma}
\label{lem:host}
For fixed constants $d>0$, $0<\eps,\eta<1/2$, $0<\xi\le1$, and
$q_*\in(0,1/2)$, there is
$\eps_0=\eps_0(d,\eps)\in(0,\eps)$ such that, for every fixed integer $t_1\ge1$,
there are $c,C,L>0$ with the following property.  If
$cn^{-1/2}\le p=p(n)<1$, then a.a.s.\ $\Gamma=\Gnp$ has
all the properties below simultaneously.
The constants are independent of $n$ and $p$; in particular, $\eps_0$
is chosen before the upper bound $t_1$ on the number of clusters.
\begin{enumerate}[label=(\roman*),leftmargin=*]
\item $\Delta(\Gamma)\le2pn$.  For every partition
$V(\Gamma)=V_0\mathbin{\dot\cup}V_1\mathbin{\dot\cup}\cdots
\mathbin{\dot\cup}V_t$, where $1\le t\le t_1$, $|V_0|\le\eps n$, and
$|V_i|=m$ for $i\in[t]$, all but $Cp^{-1}$ vertices $x$ satisfy,
for every $i\in[t]$,
\[
 d_\Gamma(x,V_i)\le(1+\eps)pm,
 \qquad d_\Gamma(x,V_0)\le2\eps pn.
\]

\item For every subgraph $G\subseteq\Gamma$, every partition as in (i),
and every pair $(V_i,V_j)$ with $1\le i<j\le t$ that is
$(\eps_0,d,p)$-lower-regular in $G$, all but $Cp^{-2}$ vertices $x$
have $(N_\Gamma(x,V_i),N_\Gamma(x,V_j))$ $(\eps,d,p)$-lower-regular
in $G$.

\item For every partition as in (i) and every union $A$ of at least
$\xi t$ clusters among $V_1,\ldots,V_t$, all but $Cn/p$ host edges $xy$
satisfy
\[
 |N_\Gamma(x,y,A)|\ge(1-\eta)p^2|A|.
\]

\item For every $A$ as in (iii), every family $\mathcal S$ of pairwise
vertex-disjoint stars in $\Gamma$ whose vertices lie in
$V(\Gamma)\setminus A$, and every assignment of parameters
$q_S\in[q_*,1-q_*]$ to its stars, the following holds.  If every
$S\in\mathcal S$ has at least $Lp^{-1}$ leaves and is
$(q_S,\eta)$-bad, then $|\mathcal S|<Lp^{-1}$.  Here a star $S$ with centre
$x$ is $(q_S,\eta)$-bad with respect to $A$ if there is a set
$D_S\subseteq N_\Gamma(x,A)$ with $|D_S|\le q_Sp|A|$ such that,
for every leaf $y$ of $S$,
\[
 d_\Gamma(y,D_S)\ge(1+\eta)q_Sp^2|A|.
\]
\end{enumerate}
Stars in (iv) are not required to be induced subgraphs of $\Gamma$.
In (i), the exceptional vertex set may depend on the partition, but it
works for all its clusters simultaneously.  In (ii), the exceptional set
may depend on $G$ and the cluster pair; in (iii), it may depend on $A$.
\end{lemma}

\begin{proof}
The maximum-degree assertion in (i) is \cite[Lemma~14(a)]{ABKR}.  To prove
the remaining assertions in (i) and (iii), we use $\xi'=\xi(1-\eps)$
and an integer
$M>\max\{2t_1,\eps^{-1},(\xi')^{-1}\}$.
Since $0<\eps<1/2$, every cluster $V_i$, $i\ge1$, has size at least
$(1-\eps)n/t_1>n/M$.  By \cite[Lemma~14(d)]{ABKR}, for each such
cluster all but $O(p^{-1})$ vertices have at most $(1+\eps)p|V_i|$
neighbours in it.  The exceptional class $V_0$ is contained in a set
$V_0'$ of size $\lceil\eps n\rceil$, to which the same lemma applies.
Thus all but
$O(p^{-1})$ vertices have at most
$(1+\eps)p|V_0'|\le2\eps pn$ neighbours in $V_0$, for all sufficiently
large $n$.  The union of the exceptional sets for these at most
$t_1+1$ sets has size $O(p^{-1})$, which proves (i).  Since
Lemma~14(d) holds simultaneously for every eligible vertex set, the
partition may be selected after the host is exposed.

For (iii), every union $A$ of at least $\xi t$ clusters satisfies
$|A|\ge\xi tm=\xi(n-|V_0|)\ge\xi'n>n/M$.
Lemma~15 of \cite{ABKR}, with error
parameter $\eta$ and its two reference sets both equal to $A$, leaves at
most $O(n/p)$ exceptional host edges.  The two reference sets in that
lemma are not required to be disjoint.  Its
hypothesis $p=\omega(\log n/n)$ follows from $p\ge cn^{-1/2}$.  This proves
(iii), uniformly over all eligible cluster unions.

For (ii), \cite[Lemma~1.27]{ABHKP} with target parameter $\eps/2$
provides an input threshold depending only on $d$ and $\eps$.
We choose $0<\eps_0<\eps$ with $2\eps_0$ below this threshold,
independently of $t_1$.
The cited result uses non-strict inequalities in the definition of
lower-regularity.  Our $(\eps_0,d,p)$-lower-regularity implies its
$(2\eps_0,d,p)$-lower-regularity, since subsets of relative size at least
$2\eps_0$ have relative size strictly greater than $\eps_0$.
Conversely, its output with parameter $\eps/2$ implies our output with
parameter $\eps$, since the density lower bound $(d-\eps/2)p$ is strictly
larger than $(d-\eps)p$.

For an eligible pair $(V_i,V_j)$ with $|V_i|=|V_j|=m$,
the cited inheritance lemma requires
\[
 m\ge C_0\max\{p^{-2},p^{-1}\log n\},
\]
where $C_0>0$ depends only on $d$ and $\eps$.  Under this condition,
it provides an exceptional vertex set $B_{ij}\subseteq V(\Gamma)$ with
\[
 |B_{ij}|\le C_0\max\{p^{-2},p^{-1}\log(\e n/m)\}.
\]
For every $x\in V(\Gamma)\setminus B_{ij}$, the pair
$(N_\Gamma(x,V_i),N_\Gamma(x,V_j))$ is
$(\eps,d,p)$-lower-regular in $G$, by the parameter conversion above.

Since $m\ge(1-\eps)n/t_1$, the cluster-size condition holds after enlarging
$c$ so that $c^2\ge2C_0t_1/(1-\eps)$.  Indeed, its $p^{-2}$ term is then at most
$(1-\eps)n/(2t_1)$, while
\[
 C_0p^{-1}\log n\le\frac{C_0}{c}\sqrt n\log n=o(n).
\]
For sufficiently large $n$, this second term is also at most
$(1-\eps)n/(2t_1)$, so both terms in the maximum are at most $m$.
For the exceptional-set bound, the inequalities
$1\le\log(\e n/m)\le\log(\e t_1/(1-\eps))$ and $p^{-1}\le p^{-2}$ give
\[
 |B_{ij}|\le C_0\log\!\left(\frac{\e t_1}{1-\eps}\right)p^{-2}
 =O_{d,\eps,t_1}(p^{-2}).
\]
The high-probability event in the cited lemma holds simultaneously for
all subgraphs $G$ and all eligible pairs of vertex sets, so no additional
probability union bound over subgraphs or partitions is needed.
For each fixed $G$ and partition, taking the union of the sets $B_{ij}$
over all eligible cluster pairs gives a common exceptional set whose
size is at most $\binom{t_1}{2}$ times the preceding bound, and hence is
$O_{d,\eps,t_1}(p^{-2})$.  Increasing $C$ proves (ii).

\smallskip
We prove (iv) simultaneously for all $A$ with $|A|\ge\xi'n$.
For a constant $L\ge2$ whose value will be specified below, we write
$r=s_0=\lceil L/(2p)\rceil$.  It suffices to exclude $r$ disjoint bad
stars with exactly $s_0$ leaves, allowing a different admissible parameter
for each star.  Indeed, $L/(2p)\ge1$ gives
$r=s_0\le L/(2p)+1\le L/p$, so any family forbidden by (iv) contains
$r$ stars, each with at least $s_0$ leaves; each star remains bad with
the same parameter and witness after its leaf set is reduced to size $s_0$.
If $r(s_0+1)>n$, there is nothing to prove.

For fixed $A$ and $r$ distinct centres outside $A$, we expose all edges
incident with the centres and restrict attention to exposures in which
each centre has degree at most $2pn$.  This restriction depends only on
the exposed edges.  For each centre, there are then at most $2^{2pn}$
choices for its leaf set and at most $2^{2pn}$ choices for its witness
$D\subseteq N_\Gamma(x,A)$.  Among choices giving vertex-disjoint stars
outside $A$, all edges from the leaves to the witnesses are
unexposed and mutually independent, since no leaf or witness is a centre,
each leaf belongs to only one star, and all witnesses lie in $A$.
This remains true even if witnesses overlap, because distinct leaves
give distinct leaf--witness edges.

For a fixed witness $D$, we write
$\widehat q(D)=\max\{q_*,|D|/(p|A|)\}$ and
$m_D=\widehat q(D)p^2|A|$.  If $D$ witnesses badness for any admissible
$q$, then $\widehat q(D)\le q$.  Thus every leaf must have at least
$(1+\eta)m_D$ neighbours in $D$.  Conditional on the exposure, its degree
into $D$ is binomial with mean $\mu_D=p|D|\le m_D$.  The additive
Chernoff bound gives
\[
 \Pr\bigl(d_\Gamma(y,D)\ge(1+\eta)m_D\mid\text{exposed edges}\bigr)
 \le \exp\left\{-\frac{((1+\eta)m_D-\mu_D)^2}
 {(1+\eta)m_D+\mu_D}\right\}
 \le \exp\{-\tfrac25\eta^2m_D\},
\]
where the last inequality uses $\mu_D\le m_D$ and
$(2+\eta)^{-1}\ge2/5$.  Since $m_D\ge q_*p^2\xi'n$, independence over all
$rs_0$ leaves bounds the probability for fixed leaf sets and witnesses
by $\exp\{-\tfrac25\xi'\eta^2q_*p^2nrs_0\}$.
The parameter $q$ has been eliminated in favour of $D$, so this estimate
allows a different $q$ for each star without any further union bound.
Thus no discretisation of the interval $[q_*,1-q_*]$ is needed.

For each fixed $A$ and choice of centres, the conditional union bound
over leaf sets and witnesses is uniform over all exposures satisfying
the centre-degree restriction, so it also bounds the unconditional
probability of a bad family together with this restriction.
Since $\Delta(\Gamma)\le2pn$ implies the restriction, the union bound
over $A$ and centres bounds the probability of
such a family together with $\Delta(\Gamma)\le2pn$ by
\[
 2^n n^r 2^{4pnr}
 \exp\{-\tfrac25\xi'\eta^2q_*p^2nrs_0\}.
\]
As $L/(2p)\le r=s_0\le L/p$, its logarithm is at most
\[
 n\log2+\frac Lp\log n+4Ln\log2
       -\frac{\xi'\eta^2q_*L^2}{10}\,n.
\]
We choose $L$ large enough that
$\xi'\eta^2q_*L^2/10>\log2+4L\log2+2$, which is possible because the
left-hand side is quadratic in $L$.  Since $p\ge cn^{-1/2}$ gives
$p^{-1}\log n\le c^{-1}\sqrt n\log n=o(n)$, the last display is at most
$-n$ for large $n$.
The failure probability in (iv) is therefore at most
$\e^{-n}+\Pr(\Delta(\Gamma)>2pn)=o(1)$, since excluding such a family
gives $|\mathcal S|<r\le L/p$.  In particular, we have not conditioned the
unexposed edges on the global maximum-degree event.
For sufficiently large $c$ and $C$, the intersection of the finitely
many high-probability events proves all four assertions simultaneously.
\end{proof}

The following form of sparse regularity records both the reduced minimum
degree and the vertexwise degree bound for the same partition.

\begin{lemma}\label{lem:regularity}
For fixed $0<\beta,d,\eps_0<1$ and an integer $t_0\ge1$, there are $c>0$
and an integer $t_1\ge t_0$, depending only on $\beta,d,\eps_0,t_0$,
such that, if $cn^{-1/2}\le p=p(n)<1$, a.a.s.\ every spanning triangle-free $H\subseteq\Gnp$ with
$\delta(H)\ge\beta pn$ has a partition
$V(H)=V_0\mathbin{\dot\cup}V_1\mathbin{\dot\cup}\cdots
\mathbin{\dot\cup}V_t$ and a
triangle-free graph $R$ on $[t]$ with $t_0\le t\le t_1$,
$|V_0|\le\eps_0n$, and $|V_1|=\cdots=|V_t|$.
For each $ij\in E(R)$, the pair $(V_i,V_j)$ is
$(\eps_0,d,p)$-lower-regular in $H$, and
$\delta(R)\ge(\beta-d-\eps_0)t$.  Moreover, for every $i\in[t]$ and
$x\in V_i$,
\[
 d_H\Bigl(x,\bigcup_{\substack{j\in[t],ij\notin E(R)}}V_j\Bigr)
 \le(d+\eps_0)pn.
\]
The union includes $V_i$ itself but not $V_0$.
\end{lemma}

\begin{proof}
We write $\Gamma=\Gnp$ and fix
$0<\theta<\min\{\beta,\eps_0/2\}$.  Then
$\delta(H)\ge\beta pn>(\beta-\theta)pn$, so the strict minimum-degree
hypothesis in \cite[Lemma~10]{ABKR} is satisfied with degree parameter
$\beta-\theta$.  We follow its construction while retaining the
vertexwise conclusion of \cite[Lemma~7]{ABKR}.
We first choose $0<\eps_1<\min\{\theta,d/4,1/4\}$ sufficiently small
for \cite[Lemma~8]{ABKR} with inheritance parameter $d/8$ and density
parameter $d$, and denote by $C_0$ its common constant in the
cluster-size hypothesis and the exceptional-set bound.
The upper bound $t_1$ supplied by
\cite[Lemma~7]{ABKR} with degree
parameter $\beta-\theta$, regularity parameter $\eps_1$, density cutoff
$d$, and lower bound $t_0$ is therefore fixed before we choose $c$.

The upper-density hypothesis of \cite[Lemma~7]{ABKR} holds simultaneously
for all $H\subseteq\Gamma$ on a high-probability host event, by
\cite[Lemma~14(c)]{ABKR}, applied with error $\eps_1^2/1000$ at the
fixed linear set-size scale $\eps_1n/t_1$.  Indeed,
$p\ge cn^{-1/2}$ implies $p=\omega(\log n/n)$.
Lemma~7 gives the required partition and
an $(\eps_1,d,p)$-reduced graph $R$, together with the vertexwise bound.

If three clusters $X,Y,Z$ form a triangle in $R$, we write
$m=|X|=|Y|=|Z|$ and derive a contradiction.
We have $m\ge(1-\eps_1)n/t_1\ge n/(2t_1)$.  We choose $c$ so large
that $C_0/c^2<1/(8t_1)$.  Since $p\ge cn^{-1/2}$, for all sufficiently
large $n$ we have
\[
 C_0\max\{p^{-2},p^{-1}\log n\}
 \le C_0\max\{n/c^2,\sqrt n\log n/c\}
 <\frac{n}{8t_1}\le\frac m4.
\]
Thus the cluster-size hypothesis of \cite[Lemma~8]{ABKR} holds, and
outside fewer than $m/4$ exceptional vertices $z$, the pair
$(N_\Gamma(z,X),N_\Gamma(z,Y))$ is $(d/8,d,p)$-lower-regular in $H$.
Moreover, \cite[Lemma~14(d)]{ABKR} gives
$d_\Gamma(z,X),d_\Gamma(z,Y)\le2pm$ outside $O(p^{-1})=o(n)$ vertices.
For sufficiently large $n$, these additional exceptions number fewer
than $m/4$, so the union of the exceptional sets has size less than
$m/2$.  Lower-regularity of $(Z,X)$ and
$(Z,Y)$ gives $d_H(z,X),d_H(z,Y)\ge(d/2)pm$ for all but at most
$2\eps_1m<m/2$ vertices of $Z$, so some $z\in Z$ satisfies all three
requirements.  Indeed, more than $\eps_1m$ vertices of $Z$ with
fewer than $(d/2)pm$ neighbours in $X$ would form, together with $X$,
a pair of density less than $(d/2)p<(d-\eps_1)p$, contradicting
lower-regularity; the same argument applies to $Y$.
Each of $N_H(z,X)$ and $N_H(z,Y)$ then occupies at least
a $d/4$ fraction of the corresponding host neighbourhood, strictly
exceeding the inherited regularity parameter $d/8$.  Since
$d-d/8>0$, the inherited lower-regularity gives an edge of $H$ between these two sets,
which forms a triangle with $z$, a contradiction.
This argument leaves the partition and $R$ unchanged, so the vertexwise
bound from Lemma~7 remains valid.
Consequently,
\[
 \delta(R)\ge(\beta-\theta-d-\eps_1)t
       \ge(\beta-d-2\theta)t
       \ge(\beta-d-\eps_0)t,
\]
and, for every $x\in V_i$, the number of its neighbours in clusters
not adjacent to $i$ in $R$ is at most
$(d+\eps_1)pn\le(d+\eps_0)pn$.
Since $\eps_1\le\eps_0$, the lower-regularity of each retained pair
and the bound on $|V_0|$ also hold with parameter $\eps_0$.
Here we keep $R$ fixed; we do not add edges for other pairs that may
become lower-regular with the larger parameter $\eps_0$.
All host events used above hold simultaneously for all eligible vertex
sets and subgraphs.  The choices of $\theta$, $\eps_1$, $t_1$, and $c$
depend only on $\beta,d,\eps_0,t_0$, as required.
\end{proof}

\section{Homomorphisms after edge deletions}\label{sec:repair}

We use the random-graph properties from \cref{lem:host} to prepare a
triangle-free subgraph for the compression lemma.  Given a sparse regular
partition with clusters $V_1,\ldots,V_t$ of common size $m$, we assign
each typical vertex $x$ the profile
$I(x)=\{i\in[t]:d_H(x,V_i)>\tau pm\}$, where $\tau>0$ is a fixed
constant chosen below.  
We show that deleting $O(n/p)$ edges leaves only edges joining vertices
whose profiles are disjoint independent sets in the reduced graph.
\Cref{lem:compression} then gives the required homomorphism.

\begin{lemma}
\label{lem:repair}
For fixed $0<\rho<\gamma\le1/10$, there are $c,C>0$, depending only on
$\rho$ and $\gamma$, such that, whenever
$cn^{-1/2}\le p=p(n)<1$, a.a.s.\ every spanning
triangle-free $H\subseteq\Gnp$ with
$\delta(H)\ge(1/3+\gamma)pn$ has a subgraph obtained by deleting at most
$Cn/p$ edges that admits a homomorphism to an Andr\'{a}sfai or Vega graph
$T$ with a probability vector $w$ on $V(T)$ such that, for every
$u\in V(T)$,
\[
 w(N_T(u))>1/3+\rho.
\]
\end{lemma}

\begin{proof}
We mark edges for deletion and leave $H$ unchanged until the final step.
Thus all degree estimates below refer to the original graph $H$.

We choose $0<q_*<(\gamma-\rho)/2$ and $0<\eta<1/2$ such that
\begin{equation}
 (1+\eta)(1-3\rho)<1-\eta.                      \label{eq:eta-choice}
\end{equation}
For example, $\eta=\rho$ satisfies this inequality.
With $\tau=10d$ and $\xi=3\rho$, the remaining parameters $d,\eps>0$
are chosen to satisfy
\begin{equation}
 \eps<\min\{d,1/2\},\quad
 \tau+4\eps<(\gamma-\rho)/2,\quad
 \tau(1/3+\rho)(1-\eps)>d+\eps.
 \label{eq:hierarchy}
\end{equation}
These conditions hold if $d$ is sufficiently small and then $\eps$ is
sufficiently small relative to $d$.
The inheritance parameter $\eps_0\le\eps$ supplied by \cref{lem:host}
for $d,\eps,\eta,\xi,q_*$ determines an upper bound $t_1$ through
\cref{lem:regularity}, with $\beta=1/3+\gamma$, regularity parameter
$\eps_0$, density cutoff $d$, and $t_0=1$.  Once $t_1$ is fixed, we
choose $c$ sufficiently large for both lemmas and use the constant $L$
from \cref{lem:host}(iv).
The remainder of the proof takes place on the intersection of these two
high-probability events for $\Gamma=\Gnp$, with $H\subseteq\Gamma$
satisfying the assumptions.
These events hold uniformly over the eligible subgraphs and partitions,
so they also apply to the profile-dependent sets chosen below.
\Cref{lem:regularity} gives a partition
$V(H)=V_0\mathbin{\dot\cup}V_1\mathbin{\dot\cup}\cdots
\mathbin{\dot\cup}V_t$ with $t\le t_1$, $|V_0|\le\eps n$, and
$|V_1|=\cdots=|V_t|=m$, together with a triangle-free reduced graph $R$
satisfying
\begin{equation}
 \delta(R)>(1/3+\rho)t.                         \label{eq:R-margin}
\end{equation}
Indeed, \cref{lem:regularity} gives
$\delta(R)\ge(1/3+\gamma-d-\eps_0)t$, which is stronger than
\eqref{eq:R-margin} by \eqref{eq:hierarchy}.  For the same partition, it also gives, for every
$i\in[t]$ and $x\in V_i$,
\begin{equation}
 d_H\Bigl(x,\bigcup_{\substack{j\in[t], ij\notin E(R)}}V_j\Bigr)
 \le(d+\eps_0)pn
 \le(d+\eps)pn.                                  \label{eq:degree-form}
\end{equation}

The exceptional set $W$ in \cref{lem:host}(i) has size $O(p^{-1})$.
We mark every edge incident with $W$, at a cost of
$O(p^{-1}\cdot pn)=O(n)$ edges because $\Delta(\Gamma)\le2pn$.
For $x\notin W$, its \emph{profile} is
\[
 I(x)=\{i\in[t]:d_H(x,V_i)>\tau pm\},
\]
and the corresponding profile
classes are $N_I=\{x\notin W:I(x)=I\}$ for $I\subseteq[t]$.
Every profile $I$ with $N_I\ne\varnothing$ satisfies
\begin{equation}
 |I|>\left(\frac13+\frac{\gamma+\rho}{2}\right)t
     >(1/3+\rho)t.                               \label{eq:profile-margin}
\end{equation}
Indeed, if $x\in N_I$ and
$|I|\le(1/3+(\gamma+\rho)/2)t$, we obtain a contradiction as follows.
The degree bounds in \cref{lem:host}(i) and the definition of $I$ give
\begin{align*}
 d_H(x)
 &\le(1+\eps)pm|I|+2\eps pn+\tau pm(t-|I|)\\
 &\le pm|I|+(\eps+\tau)pmt+2\eps pn\\
 &\le\bigl(1/3+(\gamma+\rho)/2+\eps+\tau\bigr)pmt+2\eps pn\\
 &\le\bigl(1/3+(\gamma+\rho)/2+\tau+3\eps\bigr)pn\\
 &<(1/3+\gamma)pn.
\end{align*}
Here the second inequality uses $|I|\le t$ and $t-|I|\le t$,
the third uses the assumed upper bound on $|I|$, and the fourth uses
$mt\le n$.  The final inequality follows from $\tau=10d$ and
\eqref{eq:hierarchy}.  This contradicts the minimum-degree assumption.

For each $I\subseteq[t]$, we write $S_I=\bigcup_{i\in I}V_i$.
If $I$ is independent in $R$, then
\begin{equation}
 N_I\cap S_I=\varnothing.                       \label{eq:outside}
\end{equation}
Indeed, if $x\in N_I\cap V_i$ with $i\in I$, then independence of $I$ and
\eqref{eq:degree-form} give $d_H(x,S_I)\le(d+\eps)pn$, whereas
\eqref{eq:profile-margin} and the definition of $I$ give
\begin{align*}
 d_H(x,S_I)&>\tau pm|I|
 >\tau(1/3+\rho)pmt
 \ge\tau(1/3+\rho)(1-\eps)pn
  >(d+\eps)pn,
\end{align*}
by \eqref{eq:hierarchy}, a contradiction.

The vertices outside $W$ whose profiles are not independent in $R$ form
a set $B$, and we mark all edges incident with this set.
If $ij\in E(R[I])$ and $x\in N_I$, then $x$ must be a two-sided
inheritance exception for $(V_i,V_j)$.  Otherwise the two sets
$N_H(x,V_i)$ and $N_H(x,V_j)$ each occupy more than a
$\tau/(1+\eps)$ proportion of the corresponding host neighbourhood, because
$d_H(x,V_i),d_H(x,V_j)>\tau pm$ while
$d_\Gamma(x,V_i),d_\Gamma(x,V_j)\le(1+\eps)pm$.  The hierarchy gives
$\tau/(1+\eps)>\eps$ and $d-\eps>0$.  The inherited
$(\eps,d,p)$-lower-regular pair therefore contains an $H$-edge between
these two sets, producing a triangle with $x$, contradicting $H$'s triangle-freeness.
For each $ij\in E(R)$, let $E_{ij}$ be the set of vertices $x$ for which
$(N_\Gamma(x,V_i),N_\Gamma(x,V_j))$ is not
$(\eps,d,p)$-lower-regular in $H$.  Applying \cref{lem:host}(ii) with
$G=H$ gives $|E_{ij}|\le Cp^{-2}$ for every such pair.
Every $x\in B$ has a profile $I(x)$ containing an edge $ij$ of $R$,
and the preceding argument shows that $x\in E_{ij}$.  Hence $B\subseteq\bigcup_{ij\in E(R)}E_{ij},$ and
\[
 |B|\le\sum_{ij\in E(R)}|E_{ij}|
 \le e(R)Cp^{-2}
 \le\binom{t_1}{2}Cp^{-2}=O(p^{-2}).
\]
Here $C$ and $t_1$ are independent of $n$ and $p$, and the estimate does
not require the sets $E_{ij}$ to be disjoint.  Since
$\Delta(\Gamma)\le2pn$, the edges incident with $B$ cost at most
$|B|\Delta(\Gamma)=O(p^{-2}\cdot pn)=O(n/p)$.

It remains to control edges between intersecting independent profiles.
For occurring independent profiles $I,J$ with $I\cap J\ne\varnothing$,
including the case $I=J$, we write $K=I\cap J$ and $A=S_K$.  For $k\in K$,
independence shows that $I\cup J\subseteq[t]\setminus N_R(k);$ \eqref{eq:R-margin} and
\eqref{eq:profile-margin} give
\begin{equation}
 |I\cup J|<(2/3-\rho)t,
 \quad |K|=|I|+|J|-|I\cup J|>3\rho t.                               \label{eq:overlap-size}
\end{equation}
Thus $A$ is a union of more than $\xi t$ clusters, as required by
\cref{lem:host}(iii) and (iv).  By \eqref{eq:outside}, every vertex of
$N_I\cup N_J$ lies outside $A$.

The parameters $q_X=(|X|-(1/3+\rho)t)/|K|$, for $X\in\{I,J\}$,
measure the excess profile size relative to the overlap.
By \eqref{eq:profile-margin} and $|K|\le t$, we have
$q_X>(\gamma-\rho)/2>q_*$.  Moreover, \eqref{eq:overlap-size} gives
\begin{equation}
 \begin{aligned}
  q_I+q_J
  &=1+\frac{|I\cup J|-(2/3+2\rho)t}{|K|}
  <1-\frac{3\rho t}{|K|}\le1-3\rho.
 \end{aligned}                                  \label{eq:qsum-bound}
\end{equation}
Since each parameter exceeds $q_*$ and their sum is less than $1$,
each is also less than $1-q_*$.
Thus both parameters lie in the interval covered by
\cref{lem:host}(iv), and \eqref{eq:eta-choice} gives
\begin{equation}
 (1+\eta)(q_I+q_J)<1-\eta.                         \label{eq:qsum}
\end{equation}

For every $x\in N_I\cup N_J$, the set
$D_x=N_\Gamma(x,A)\setminus N_H(x,A)$ records its missing host
neighbours in $A$.  We claim that $|D_x|\le q_Xp|A|$ for
$X\in\{I,J\}$ and $x\in N_X$.
If, for example, $u\in N_I$ and $|D_u|>q_Ip|A|$, since $A\subseteq S_I$, every vertex of $D_u$ is
a host neighbour of $u$ in $S_I$ but not an $H$-neighbour, so
$d_H(u,S_I)\le d_\Gamma(u,S_I)-|D_u|$.
Since $u\notin W$, \cref{lem:host}(i) gives
$d_\Gamma(u,S_I)\le(1+\eps)pm|I|$ and
$d_H(u,V_0)\le2\eps pn$.  For each $i\notin I$, the definition of the
profile gives $d_H(u,V_i)\le\tau pm$.  Splitting the degree over these
three parts and using $|A|=m|K|$ and
$q_I|K|=|I|-(1/3+\rho)t$, we obtain
\begin{align*}
 d_H(u)
 &=d_H(u,S_I)+\sum_{i\in[t]\setminus I}d_H(u,V_i)+d_H(u,V_0)\\
 &\le(1+\eps)pm|I|-|D_u|+\tau pm(t-|I|)+2\eps pn\\
 &<(1+\eps)pm|I|-pm\bigl(|I|-(1/3+\rho)t\bigr)
       +\tau pm(t-|I|)+2\eps pn\\
 &\le(1/3+\rho+\eps+\tau)pmt+2\eps pn\\
 &\le(1/3+\rho+\tau+3\eps)pn\\
 &<(1/3+\gamma)pn.
\end{align*}
The final inequality follows from
\eqref{eq:hierarchy}.
This contradicts the minimum-degree assumption.
The same argument applies to vertices in $N_J$.

We now orient the nonexceptional edges so that every incoming star is
bad, and then use the bound on vertex-disjoint bad stars to show that
there are only $O(n/p)$ edges to delete for this profile pair.
There are $O(n/p)$ host edges that fail the codegree estimate in
\cref{lem:host}(iii) for $A$.  We bound these separately and orient
each remaining $H$-edge $uv$ with $u\in N_I$ and $v\in N_J$;
when $I=J$, the endpoint labels are arbitrary.
Triangle-freeness gives $N_\Gamma(u,v,A)\subseteq D_u\cup D_v$.
The edge $uv$ is oriented towards $u$ if
\[
 |N_\Gamma(u,v,A)\cap D_u|
 \ge(1+\eta)q_Ip^2|A|,
\]
and towards $v$ otherwise.  In the second case, the codegree bound and
\eqref{eq:qsum} give
\begin{align*}
 |N_\Gamma(u,v,A)\cap D_v|
 &\ge |N_\Gamma(u,v,A)|-|N_\Gamma(u,v,A)\cap D_u|\\
 &>\bigl((1-\eta)-(1+\eta)q_I\bigr)p^2|A|\\
 &>(1+\eta)q_Jp^2|A|.
\end{align*}
Since $D_x\subseteq N_\Gamma(x,A)$, every incoming neighbour $y$ of $x$
satisfies $d_\Gamma(y,D_x)=|N_\Gamma(x,y,A)\cap D_x|$.
For $X\in\{I,J\}$ and $x\in N_X$, the orientation and the bound on
$|D_x|$ therefore make every incoming star centred at $x$
$(q_X,\eta)$-bad, with witness $D_x$.  All its vertices lie outside $A$,
so \cref{lem:host}(iv) applies.

With $s=\lceil Lp^{-1}\rceil$ and $L$ from \cref{lem:host}(iv),
a greedy selection gives a maximal family of vertex-disjoint incoming
stars with exactly $s$ leaves.  By that assertion, the family
has fewer than $L/p$ members, even though the parameter may differ
between stars centred in $N_I$ and in $N_J$.
The number of edges incident with their vertices is at most
$O(p^{-1})(s+1)\Delta(\Gamma)=O(n/p)$.
The oriented subgraph on the remaining vertices has indegree less than
$s$ at every vertex by maximality.  Thus at most
$sn=O(n/p)$ oriented edges remain.  If $I\ne J$, every edge between
$N_I$ and $N_J$ is either a codegree exception, a nonexceptional edge
incident with a vertex of the selected stars, or a remaining oriented
edge.  Adding the three bounds above therefore gives
\begin{equation}
 e_H(N_I,N_J)
 \le Cn/p+(L/p)(s+1)\Delta(\Gamma)+sn
 =O(n/p).                                       \label{eq:overlap-bound}
\end{equation}
If $I=J$, the same argument gives $e(H[N_I])=O(n/p)$.
Thus all edges between intersecting independent profiles can be marked
at a cost of $O(n/p)$ per profile pair.

We therefore mark all edges between intersecting independent profiles,
including edges within each profile class.  By \eqref{eq:overlap-bound} and its
within-class counterpart, each profile pair contributes $O(n/p)$ edges.
As there are at most $2^t\cdot2^t=4^t\le4^{t_1}$ ordered
profile pairs, these edges cost $O(n/p)$ in total.  Together with
the $O(n)$ edges incident with $W$ and the $O(n/p)$ edges incident with
$B$, the total number of marked edges is $O(n/p)$.

The spanning graph $H^\circ$ obtained by deleting the marked edges
has edges only between vertices with disjoint independent profiles.
By \eqref{eq:R-margin} and \eqref{eq:profile-margin},
\cref{lem:compression} applies to $R$ and the family of occurring
independent profiles, giving the auxiliary graph $F_0$ and a
homomorphism $\pi:F_0\to T$.  Define $f(x)=z_{I(x)}$ for
$x\notin W\cup B$.  Every edge $xy\in E(H^\circ)$ has both endpoints
outside $W\cup B$ and satisfies $I(x)\cap I(y)=\varnothing$, so
$z_{I(x)}z_{I(y)}\in E(F_0)$ by the definition of $F_0$.
Since the vertices in $W\cup B$ are isolated, assigning them arbitrary
images in $V(F_0)$ extends $f$ to a homomorphism $H^\circ\to F_0$.
Thus the required homomorphism is the composition
\[
 H^\circ\xrightarrow{\,f\,}F_0\xrightarrow{\,\pi\,}T.
\]
The probability vector $w$ supplied by the same lemma
has the required neighbourhood weights.  All constants
depend only on $\rho$ and $\gamma$.
\end{proof}

\section{Proofs of the main results}
\label{sec:main-proofs}

The repair lemma and the discreteness of the possible template
complexities together yield the structural theorem, from which the
edit-chromatic profile follows by excluding templates whose chromatic
number exceeds the prescribed number of parts.  The lower-bound constructions used to
prove $\alpha_q^{\rm edit}\ge\alpha_q$ are given in
\cref{sec:optimality}.

\begin{proof}[Proof of \cref{thm:template}]
The possible template complexities form the set
\[
 \calL=\{3k-1:k\ge1\}\cup
 \{27i-19-3\mu-3\nu:i\ge2,\ \mu,\nu\in\{0,1\}\}.
\]
Since $\calL$ is an unbounded set of positive integers, it has a least
member $L_+$ strictly larger than $(3\gamma)^{-1}$.  The inequality
$(3L_+)^{-1}<\gamma$ therefore allows us to choose $\rho$ such that
\begin{equation}
 \frac1{3L_+}<\rho<\gamma.                         \label{eq:rho-choice}
\end{equation}
We fix this choice of $\rho$ as a function of $\gamma$ alone, so all
constants supplied by \cref{lem:repair} depend only on $\gamma$.
For the constant $c$ supplied by \cref{lem:repair} for $\rho,\gamma$
and $cn^{-1/2}\le p<1$, that lemma and \cref{lem:certificate} give
$L(T)<(3\rho)^{-1}<L_+$.  Since
$L(T)\in\calL$ and $L_+$ is the first member of $\calL$ above
$(3\gamma)^{-1}$, we obtain $L(T)\le(3\gamma)^{-1}$.

For $p=1$, the bound $\delta(H)\ge(1/3+\gamma)n>(1/3+\rho)n$
allows us to apply \cref{lem:compression} with $R=H$, $t=n$, and
$\calI=\varnothing$, obtaining a homomorphism $H\to T$ and a
probability vector $w$ satisfying \eqref{eq:weighted-margin}.
The same application of
\cref{lem:certificate} gives $L(T)\le(3\gamma)^{-1}$ without deleting
any edges.

It remains to handle $p<cn^{-1/2}$.  If $p\le2/n$, then
$e(\Gnp)\le\binom n2\le n/p$ deterministically.  If
$2/n<p<cn^{-1/2}$, the mean $p\binom n2$ is larger than $n-1$,
so Chernoff's inequality gives $e(\Gnp)\le pn^2$ with failure
probability $\exp\{-\Omega(n)\}$.  Since
$pn^2=(np^2)n/p<c^2n/p$, we have $e(\Gnp)=O(n/p)$.
Thus deleting every edge costs $O(n/p)$ and leaves an empty graph that
maps to $\Gamma_1=K_2$, whose complexity satisfies
$L(K_2)=2\le(3\gamma)^{-1}$ because $\gamma\le1/10$.
A sufficiently large choice of $C$ covers all three ranges.
For an arbitrary sequence $p=p(n)$, the corresponding subsequences
satisfy the estimates above, so the failure probability tends to zero
along the full sequence, completing the proof.
\end{proof}

\begin{proof}[Proof of \cref{cor:profile}]
We first prove the $Cn/p$ bound using \cref{thm:template}.

For $q=2$, we choose
$\gamma'\in(1/15,\min\{1/10,1/15+\gamma\}]$ and apply
\cref{thm:template} with margin $\gamma'$, which is permitted because
$1/3+\gamma'\le2/5+\gamma$.  The resulting bound
$L(T)\le(3\gamma')^{-1}<5$ leaves only $\Gamma_1=K_2$, so the repaired
graph is bipartite.  For $q=3$, a choice of
$\gamma'\in(1/87,\min\{1/10,1/87+\gamma\}]$ similarly gives
$1/3+\gamma'\le10/29+\gamma$, and \cref{thm:template} yields
$L(T)\le(3\gamma')^{-1}<29$.  Since every Vega graph has complexity
at least $29$, the repaired graph maps to an Andr\'{a}sfai graph and is
three-colourable.  For $q\ge4$, the choice
$\gamma'=\min\{\gamma,1/10\}$ ensures that the given minimum-degree
bound is at least $(1/3+\gamma')pn$, so \cref{thm:template} gives a
repaired graph mapping to an Andr\'{a}sfai or Vega graph.
Every such template is four-colourable and hence $q$-colourable,
which completes the $Cn/p$ bound for every $\gamma>0$.

For every graph $G$, a uniformly random $q$-colouring gives
$\del_q(G)\le e(G)/q$ by taking expectations.  If
$p>n^{-7/4}$, then $pn^2>n^{1/4}\to\infty$, and Chernoff's inequality gives
$e(\Gnp)\le(1/2+o(1))pn^2$.  Thus, for all sufficiently large $n$,
\[
 \del_q(H)\le\frac1q e(\Gnp)
 \le\left(\frac1{2q}+o(1)\right)pn^2
 \le\left(\frac1{2q}+\gamma\right)pn^2.
\]
If $p\le n^{-7/4}$, the expected number
of two-edge paths in $\Gnp$ is $O(n^3p^2)=O(n^{-1/2})=o(1)$.
Markov's inequality
therefore shows that a.a.s.\ there is no such path;
equivalently, the host has maximum degree at most one and is already
bipartite, so $\del_q(H)=0$ for every $H\subseteq\Gnp$ and $q\ge2$.
Intersecting this a.a.s.\ event (or the edge-count event in the other
range) with the event giving the $Cn/p$ bound, both upper bounds hold
simultaneously for every admissible $H$.  Taking their minimum proves
\eqref{eq:profile-upper}; the two bounds need not be witnessed by the
same colouring.  Since the split is made pointwise in $n$,
the two estimates apply on the corresponding subsequences even if
$p(n)$ crosses $n^{-7/4}$ infinitely often, and their failure
probabilities tend to zero along the full sequence.

The preceding argument gives $\alpha_q^{\rm edit}\le\alpha_q$.  The
reverse inequality follows from \cref{thm:sharpness}(i), proved in
\cref{sec:optimality}.  Indeed, for any proposed threshold
$\alpha<\alpha_q$, we can choose $\eta,\eps>0$ with
$\alpha+\eps<\alpha_q-\eta$.  With $p=1$, \cref{thm:sharpness}(i)
then gives, for all sufficiently large $n$, a
triangle-free graph $H$ with $\delta(H)\ge(\alpha_q-\eta)n$
and $\del_q(H)\ge c_{q,\eta}n^2$.
It satisfies the minimum-degree hypothesis at $\alpha+\eps$, but
$\del_q(H)>Cn$ for every fixed $C$ once $n$ is sufficiently large.
Thus $\alpha$ does not satisfy the defining property of
$\alpha_q^{\rm edit}$.
\end{proof}

For completeness, the two lower-colour cases have the following
explicit forms.  If $\delta(H)\ge(10/29+\gamma)pn$, then
\[
 \del_3(H)\le\min\{C n/p,(1/6+\gamma)pn^2\},
\]
which is an edge-deletion analogue of Jin's theorem \cite{Jin}.
Similarly, if $\delta(H)\ge(2/5+\gamma)pn$, then
\[
 \del_2(H)\le\min\{C n/p,(1/4+\gamma)pn^2\},
\]
which recovers the theorem of Allen et al.\ \cite[Theorem~3]{ABKR}
at the Andr\'{a}sfai--Erd\H{o}s--S\'{o}s threshold \cite{AES}.

\section{Lower bounds}\label{sec:optimality}

The first lower bound uses the standard random blow-up transfer for a fixed
triangle-free template.  This construction is closely related to the sparse
random-host construction of Allen--B\"ottcher--Kohayakawa--Roberts
\cite[Section~3 and Theorem~5]{ABKR}; we record the weighted form needed
here.  The second lower bound uses the low-density observation that, when
$np^2\to0$, only a negligible number of host edges lie in triangles
\cite[Introduction]{ABKR}.  We retain the proof of this strengthened
minimum-degree and $q$-partite-distance conclusion, which is not stated in
that reference.

\subsection{Random blow-ups}

\begin{lemma}
\label{lem:random-blowup}
For a fixed integer $q\ge2$, a fixed triangle-free graph $F$ with
$\chi(F)>q$, and a probability vector $w:V(F)\to(0,1)$,
we write $a=\min_{x\in V(F)}w(N_F(x))$.
There is a constant $c=c(F,w,q)>0$ such that, if
$pn/\log n\to\infty$, then a.a.s.\ $\Gnp$
contains a spanning triangle-free $H$ such that
\[
 \delta(H)=(a+o(1))pn,
 \qquad \del_q(H)\ge cpn^2.
\]
\end{lemma}

The construction and the majority-colour step below are the standard
random-blow-up argument used in \cite[Section~3 and Theorem~5]{ABKR}.
We give the short proof for completeness, while recording the weighted
form needed here.

\begin{proof}
Partition $[n]$ into classes $V_x$ with
$|V_x|=w(x)n+O(1)$, and let $H$ consist of the host edges between $V_x$
and $V_y$ whenever $xy\in E(F)$.  Since $F$ is triangle-free, so is $H$.
For $z\in V_x$, Chernoff's inequality and $pn/\log n\to\infty$ give,
uniformly over all vertices,
\[
 d_H(z)=\bigl(w(N_F(x))+o(1)\bigr)pn.
\]
Thus $\delta(H)=(a+o(1))pn$.

Put $w_*:=\min_x w(x)$.  For every $xy\in E(F)$ and every
$X\subseteq V_x$, $Y\subseteq V_y$ with
$|X|,|Y|\ge w_*n/(2q)$, Chernoff's inequality gives
\[
 \Pr\left(e_{\Gamma}(X,Y)<\frac12p|X||Y|\right)
 \le \exp\left(-\frac{w_*^2pn^2}{32q^2}\right).
\]
There are at most $4^n|E(F)|$ choices.  A union bound therefore shows
that, a.a.s., $e_\Gamma(X,Y)\ge\frac12p|X||Y|$ for all such pairs $X,Y$.
For any map $\phi:V(H)\to[q]$, assign to each $V_x$ a most frequent
colour under $\phi$.  Since $\chi(F)>q$, some
edge $xy$ receives the same colour at both ends.  The corresponding two
colour classes have size at least $w_*n/(2q)$, and hence contain at least
$w_*^2pn^2/(8q^2)$ monochromatic edges.  This proves the lemma, for example
with $c=w_*^2/(8q^2)$.
\end{proof}

Choose a fixed template $F_q$ as follows.  For $q=2$ take $C_5$;
for $q=3$ take H\"aggkvist's $29$-vertex triangle-free $4$-chromatic
$10$-regular graph; and for $q\ge4$ take a Hajnal triangle-free graph with
chromatic number greater than $q$ and minimum-degree ratio greater than
$1/3-\eta/2$.  The first two choices and their weighted degrees are
classical \cite{Haggkvist,Jin,LPRRT}, while the last follows from
Hajnal's construction \cite[Section~3]{ErdosSimonovits}.  Applying
\cref{lem:random-blowup} to the uniform weight vector gives
$\delta(H)\ge(\alpha_q-\eta)pn$ and
$\del_q(H)\ge c_{q,\eta}pn^2$ for all sufficiently large $n$.
This proves \cref{thm:sharpness}(i).

In particular, when $0<\gamma<1/87$ and $p\gg n^{-1/2}$, the H\"aggkvist
blow-up gives $\delta(H)\ge(1/3+\gamma)pn$ but
$\del_3(H)=\Omega(pn^2)=\omega(n/p)$.  Thus four colours in the question of
Allen et al.\ genuinely cannot be replaced by three.

\subsection{Sharpness of the coefficient}

\begin{lemma}
\label{lem:host-distance}
For a fixed integer $q\ge2$, if $pn\to\infty$, then a.a.s.\
\[
 \del_q(\Gnp)=\left(\frac1{2q}+o(1)\right)pn^2.
\]
\end{lemma}

The first-order estimate is the standard MAX $q$-CUT result of
Coja-Oghlan, Moore, and Sanwalani \cite{MaxQCut}.  We nevertheless give a
short proof for completeness; it is also the direct form needed here.

\begin{proof}
The upper bound follows from $\del_q(G)\le e(G)/q$ and
$e(\Gnp)=(1/2+o(1))pn^2$ a.a.s.  For a fixed $q$-colouring $\phi$ with
class sizes $n_1,\ldots,n_q$, put
$M_\phi=\sum_i\binom{n_i}{2}\ge n^2/(2q)-n/2$.  The number of
monochromatic host edges is $\operatorname{Bin}(M_\phi,p)$.  With
$\varepsilon_n=(pn)^{-1/4}$, Chernoff's inequality gives
\[
 \Pr\left(X_\phi<(1-\varepsilon_n)pM_\phi\right)
 \le \exp\bigl(-\Omega(\varepsilon_n^2pn^2)\bigr)
 =\exp\bigl(-\Omega(n\sqrt{pn})\bigr).
\]
This is $o(q^{-n})$, so a union bound over all $q^n$ colourings shows that
every colouring has $(1/(2q)-o(1))pn^2$ monochromatic edges.  Together
with the upper bound this proves the lemma.  For the dense endpoint
$p>0.99$, including $p=1$, the same estimate is elementary.
\end{proof}

The next lemma makes the low-density remark of \cite[Introduction]{ABKR}
quantitative for the minimum degree and the exact first-order $q$-partite
edit distance.  The underlying deletion idea is already present there; we
retain the following short proof for completeness because this precise
simultaneous statement is not stated there.

\begin{lemma}
\label{lem:triangle-stripping}
For a fixed integer $q\ge2$, if $pn/\log n\to\infty$ and $np^2\to0$,
then a.a.s.\
$\Gamma=\Gnp$ contains a spanning triangle-free $H$ with
\[
 \delta(H)=(1-o(1))pn,
 \qquad
 \del_q(H)=\left(\frac1{2q}+o(1)\right)pn^2.
\]
\end{lemma}

\begin{proof}
The graph $H$ obtained from $\Gamma$ by deleting every edge lying in a
triangle is triangle-free.  If $T_v$ denotes the number of triangles
containing a vertex $v$, then at most $2T_v$ edges incident with $v$ are
deleted.  It therefore suffices to show that $\max_vT_v=o(pn)$, which
controls both the degree loss and the total number of deleted edges.

The event $N_\Gamma(v)=S$ only specifies edges incident with $v$.
Conditional on this event, the edges inside $S$ therefore remain
independent Bernoulli variables with parameter $p$, so
\[
 T_v\mid\{N_\Gamma(v)=S\}
 \sim\operatorname{Bin}\left(\binom{|S|}{2},p\right).
\]
Since $p^3n^2=o(pn)$ and $\log n=o(pn)$, we may take
\[
 t_n=\left\lceil\sqrt{pn\bigl(p^3n^2+\log n\bigr)}\right\rceil,
\]
which satisfies $p^3n^2+\log n\ll t_n\ll pn$.
For every fixed $S$ with $|S|\le2pn$, the conditional mean is at most
$p\binom{|S|}{2}\le2p^3n^2$.  Uniformly over all such $S$, the
binomial upper-tail estimate gives
\[
 \Pr(T_v\ge t_n\mid N_\Gamma(v)=S)
 \le\left(\frac{2\e p^3n^2}{t_n}\right)^{t_n}
 =\exp\{-\omega(\log n)\}.
\]
Chernoff's inequality gives
$\Delta(\Gamma)\le2pn$ and $\delta(\Gamma)=(1-o(1))pn$
a.a.s.  Averaging the conditional tail bound over the possible sets
$S$ with $|S|\le2pn$ gives
$\Pr(T_v\ge t_n,\ d_\Gamma(v)\le2pn)\le\exp\{-\omega(\log n)\}$.
Consequently,
\[
 \Pr\bigl(\max_v T_v\ge t_n\bigr)
 \le \Pr\bigl(\Delta(\Gamma)>2pn\bigr)
   +\sum_v\Pr\bigl(T_v\ge t_n,\ d_\Gamma(v)\le2pn\bigr)
 =o(1).
\]
Thus $\max_vT_v<t_n=o(pn)$ a.a.s.
Consequently $\delta(\Gamma)-2t_n\le\delta(H)\le\delta(\Gamma)$,
so $\delta(H)=(1-o(1))pn$.

The deleted edge set $R=E(\Gamma)\setminus E(H)$ therefore satisfies
\[
 |R|=\frac12\sum_v\bigl(d_\Gamma(v)-d_H(v)\bigr)
 \le\sum_vT_v\le nt_n=o(pn^2).
\]
Since each colouring loses at most $|R|$ monochromatic edges,
\cref{lem:host-distance} gives
\[
 \del_q(H)\ge\del_q(\Gamma)-|R|
 =\left(\frac1{2q}-o(1)\right)pn^2.
\]
Conversely, $H\subseteq\Gamma$ implies
$\del_q(H)\le\del_q(\Gamma)=(1/(2q)+o(1))pn^2$.  The two bounds prove the
claimed estimate.
\end{proof}

For $0<\gamma<1-\alpha_q$, the fixed gap $1-\alpha_q-\gamma>0$
and \cref{lem:triangle-stripping} give
$\delta(H)\ge(\alpha_q+\gamma)pn$ for all sufficiently large $n$.
The edit-distance estimate in that lemma proves
\cref{thm:sharpness}(ii) and completes the proof.

\section*{Acknowledgment}
G. Gao was supported by the National Key R\&D Program of China
(grant 2023YFA1010202), the National Natural Science Foundation of
China (grant 12401448), and the Natural Science Foundation of Fujian
Province (grant 2024J08030).  J. He was supported in part by the
Science and Technology Commission of Shanghai Municipality
(grant 22DZ2229014).

\section*{Data availability}
No datasets were generated or analysed in this study.

\section*{Declaration on the use of generative AI}

The authors used generative AI tools to assist in discussing proof strategies, checking proofs, and
improving the exposition. The authors take full responsibility for the mathematical arguments, results,
and conclusions, all of which were carefully reviewed and verified by them.

\end{document}